\documentclass{amsart}
\usepackage{mathrsfs}
\usepackage{amsfonts}
\usepackage{amssymb}
\usepackage{amsxtra}

\usepackage{color}
\usepackage{dsfont}

\usepackage[english,polish]{babel}
\newcommand{\pl}[1]{\foreignlanguage{polish}{#1}}

\theoremstyle{plain}
\newtheorem{theorem}{Theorem}
\newtheorem{proposition}{Proposition}[section]

\newtheorem{lemma}[proposition]{Lemma}

\theoremstyle{definition}
\newtheorem{definition}{Definition}[section]
\theoremstyle{remark}

\newcounter{thm}

\newcommand{\RR}{\mathbb{R}}
\newcommand{\ZZ}{\mathbb{Z}}
\newcommand{\TT}{\mathbb{T}}

\newcommand{\NN}{\mathbb{N}}

\newcommand{\Ss}{\mathbb{S}}
\newcommand{\PP}{\mathbb{P}}

\renewcommand{\atop}[2]{\substack{{#1}\\{#2}}}
\newcommand{\norm}[1]{{\left\lvert #1 \right\rvert}}

\newcommand{\abs}[1]{{\lvert {#1} \rvert}}

\newcommand{\calC}{\mathcal{C}}

\newcommand{\calF}{\mathcal{F}}

\newcommand{\calR}{\mathcal{R}}

\newcommand{\calL}{\mathcal{L}}
\newcommand{\calO}{\mathcal{O}}

\newcommand{\ind}[1]{{\mathds{1}_{{#1}}}}

\author{Ben Krause}
\address{UCLA Math Sciences Building\\
         Los Angeles CA 90095-1555}
\email{benkrause23@math.ucla.edu}

\author{Mariusz Mirek}
\address{Mariusz Mirek \\
	Universit\"{a}t Bonn \\
	Mathematical Institute\\
	Endenicher Allee 60\\
	D--53115 Bonn \\
	Germany \&
	Instytut Matematyczny\\
	Uniwersytet \pl{Wroc{\lll}awski}\\
	Pl. Grun\-waldzki 2/4\\
	50-384 \pl{Wroc{\lll}aw}\\
	Poland}
\email{mirek@math.uni-bonn.de}

\author{Bartosz Trojan}
\address{
	Bartosz Trojan\\
	Instytut Matematyczny\\
	Uniwersytet \pl{Wroc{\lll}awski}\\
	Pl. Grun\-waldzki 2/4\\
	50-384 \pl{Wroc{\lll}aw}\\
	Poland}
\email{trojan@math.uni.wroc.pl}

\title[On the Hardy--Littlewood majorant problem]
{On the Hardy--Littlewood majorant problem \\
 for arithmetic sets}

\begin{document}
\selectlanguage{english}

\begin{abstract}
  The aim of this paper is to exhibit a wide class of {sparse}
  deterministic sets, $\mathbf B \subseteq \mathbb{N}$, so that
\[ \limsup_{N \to \infty} N^{-1}|\mathbf B \cap [1,N]|= 0, \]
for which the Hardy--Littlewood majorant property holds:
\[ \sup_{|a_n|\le 1} \Big\| \sum_{n\in\mathbf B\cap[1, N]} a_n e^{2 \pi i n \xi}\Big
\|_{L^p(\mathbb{T}, {\rm d} \xi)} \leq \mathbf{C}_p \Big\|
\sum_{n\in\mathbf B\cap[1, N]} e^{2 \pi i n \xi} \Big\|_{L^p(\mathbb{T},
  {\rm d} \xi)},\] where $p \geq p_{\mathbf{B}}$ is sufficiently
large, the implicit constant $\mathbf{C}_p$ is independent of $N$, and
the supremum is taken over all complex sequences $ (a_n : n \in
\mathbb{N})$ such that $|a_n| \leq 1$.
\end{abstract}

\maketitle

\section{Introduction}
In 1937, Hardy and Littlewood \cite{HL} conjectured that for each $p \geq 2$ there
is a constant $\mathbf{C}_p > 0$ such that for every finite set $A \subset \NN$ and
every sequence $(a_n : n \in A)$ of complex numbers satisfying
$\sup_{n \in A} \abs{a_n} \leq 1$ we have
\begin{align}
	\label{eq:0}
	\Big\|
	\sum_{n\in A}
	a_n
	e^{2\pi i n\xi}
	\Big\|_{L^p(\TT, {\rm d}\xi)}
	\le
	\mathbf C_p
	\Big\|\sum_{n\in A}e^{2\pi i n\xi}\Big\|_{L^p(\TT, {\rm d}\xi)}.
\end{align}
This conjecture, known as the \emph{Hardy--Littlewood majorant problem}, was
suggested by a simple observation, based on Parseval's identity, which
implies that $\mathbf C_p=1$  for every even integer $p\ge2$. It was also noticed by Hardy and Littlewood
that $\mathbf C_3>1$. In 1962, Boas \cite{boas2} showed that $\mathbf
C_p>1$ for any
$p \not \in \{2k:k\in\NN\}$. Finally, in early seventies Bachelis
\cite{Bache} disproved the Hardy--Littlewood conjecture  showing unboundedness of
$\mathbf C_p$ for every $p \not \in\{2k:k\in\NN\}$ as $|A|\to\infty$.

Although inequality \eqref{eq:0} fails to hold in general, recently some attention has been
paid to quantify this failure. To do so, for $N \in \NN$ we consider
\[
	\mathbf C_p(N)=\sup_{A\subseteq\{1,\ldots,N\}}\mathbf C_p(A, N)
\]
where for $A \subseteq \{1,\ldots,N\}$ we have set
\[
	\mathbf C_p(A, N)
	=
	\sup_{|a_n| \le 1}
	\Big\|
	\sum_{n\in A}
	a_n e^{2\pi i n\xi}
	\Big\|_{L^p(\TT, {\rm d}\xi)}
	\cdot
	\Big\|
	\sum_{n\in A} e^{2\pi i n\xi}
	\Big\|_{L^p(\TT, {\rm d}\xi)}^{-1}.
\]
It was proven in \cite{Gerdhab} that for every $p \in (2, 4)$ there is a
constant $C>0$ such that
\[
	\log \mathbf C_p(N) \ge C \frac{\log N}{\log\log N}.
\]
Consequently, the Hardy--Littlewood majorant problem was reformulated
to a slightly weaker statement. Namely, it was conjectured that for every
$p \ge 2$ and $\varepsilon>0$ there is a constant $C_{p, \varepsilon}>0$
such that for every $N\in\NN$
\begin{align}
	\label{eq:6}
	\mathbf C_p(N)\le C_{p, \varepsilon}N^{\varepsilon}.
\end{align}
It is worth mentioning that \eqref{eq:6} implies the restriction conjecture
for the Fourier transform on $\RR^d$, i.e. that for every $p > 2d/(d-1)$ there
exists a constant $C_{p, d}>0$ such that
\begin{align}
	\label{eq:7}
	\big\|
	\widehat{f {\rm d}\sigma}
	\big\|_{L^p(\RR^d)}
	\le C_{p, d}
	\|f\|_{L^{\infty}(\Ss^{d-1}, {\rm d} \sigma)}
\end{align}
where $\sigma$ is the spherical measure on the unit sphere $\Ss^{d-1}$ in $\RR^d$.
In \cite{Gerdhab} it was stated that for suitable sets $A$ the inequality \eqref{eq:0}
may be treated as a restatement of \eqref{eq:7}. However, Mockenhaupt and Schlag \cite{MoSc}
disproved \eqref{eq:6} by showing that for all $p>2$ which is not an even integer,
there are constants $\eta>0$ and $C>0$ such that $\mathbf C_p(N) \ge C N^{\eta}$.
For $p = 3$ the same result was  obtained by Green and Ruzsa \cite{GR}.

In view of the restriction conjecture one may ask whether there are sets
$A \subseteq \{1,\ldots,N\}$ such that for every $p \ge 2$ and $\varepsilon>0$ there
exists a constant $C_{p, \varepsilon}>0$ for which we have
\begin{align}
	\label{eq:8}
	\mathbf C_p(A, N) \le C_{p, \varepsilon} N^{\varepsilon}.
\end{align}
The question above has been extensively studied by Mockenhaupt and Schlag in \cite{MoSc}
where the authors proved that for every $\varrho\in(0, 1)$ and $p\ge2$ there are
random sets $A\subseteq\{1,\ldots,N\}$ with cardinality $N^{\varrho}$ satisfying
\eqref{eq:8} with a large probability.

The Hardy--Littlewood majorant property plays an important role in combinatorial
problems. In \cite{G} Green used a variant of the inequality \eqref{eq:0} for the
set of prime numbers $\PP$ to deduce that every subset of $\PP$ with non-vanishing
relative upper-density contains at least one arithmetic progression of length three.
Specifically, Green proved that for every $p\ge2$ there is a constant $C_p>0$ such that
for all $N\in\NN$
\[
	\mathbf C_p(\PP_N, N)\le C_{p}
\]
where $\PP_N=\PP\cap[1, N]$, the set of primes less than or equal to $N$. Generally speaking, in problems of this kind it is
critical to know whether the majorant property \eqref{eq:0} holds for some $p\in(2, 3)$
with the uniform constant $\mathbf C_p$, independent of the cardinality of the set $A$
(see \cite{GT,VT}).

The present article is devoted to study a wide class of deterministic infinite sets
$A\subseteq \NN$ with vanishing Banach density, i.e.
\[
	\limsup_{N\to\infty}
	\frac{| A \cap [1,N] |}{N} = 0,
\]
and obeying the Hardy--Littlewood majorant property. In particular, we will be concerned
with the sets
\begin{align}
	\label{eq:3}
	\mathbf A=\big\{\lfloor h(n)\rfloor: n\in\NN\big\}
\end{align}
where $h$ is a regularly varying function of the form $h(x)=x \ell(x)$, for a suitably chosen
slowly varying function $\ell$, e.g.
\begin{align*}
	\ell(x)=(\log x)^B, \ \ \text{or} \ \
	\ell(x)=\exp\big({B (\log x)^C}\big), \ \ \text{or} \ \
	\ell(x)=l_m(x),
\end{align*}
where $B>0$, $C\in(0, 1)$, $l_1(x)=\log x$ and $l_{m+1}(x)=\log(l_m(x))$, for $m\in\NN$.
We show that for every $p \ge 2$ there exists a constant $C_p > 0$ such that
for every $N\in\NN$ we have
\[
	\mathbf C_p(\mathbf A_N, N)\le C_{p}
\]
where $\mathbf A_N=\mathbf A\cap[1, N]$. We also consider the sets \eqref{eq:3} with
\[
	h(x)=x^c \ell(x)
\]
for some $c>1$ sufficiently close to 1. In this case we show that it is possible to
find $p_c>2$ such that for every $p>p_c$ there exists a constant $C_{c, p}>0$ such that
for every $N\in\NN$
\[
	\mathbf C_p(\mathbf A_N, N)\le C_{c, p}.
\]
Moreover, $\lim_{c\to1} p_c=2$.

\subsection{Statement of the results}
\label{sec:2}
Before we precisely formulate the main results we need to introduce some definitions.
\begin{definition}
	\label{def:1}
	Let $\mathcal L$ be a family of slowly varying functions
	$\ell: [x_0, \infty) \rightarrow (0, \infty)$  such that
	\begin{align*}
		\ell(x)=\exp\Big(\int_{x_0}^x \frac{\vartheta (t)}{t} {\: \rm d} t \Big)
	\end{align*}
	where $\vartheta\in \calC^2([x_0, \infty))$ is a real function satisfying
	\[
		\lim_{x\to\infty}
		\vartheta(x)=0,\ \ \lim_{x\to\infty}x\vartheta'(x)=0,\ \
		\lim_{x\to\infty}x^2\vartheta''(x)=0.
	\]
\end{definition}
We also distinguish a subfamily $\calL_0$ of $\calL$.
\begin{definition}
	\label{def:2}
	Let $\calL_0$ be a family of  slowly varying functions
	$\ell:[x_0, \infty) \rightarrow (0, \infty)$ such that
	\begin{align*}
		\ell(x)=\exp\Big(\int_{x_0}^x\frac{\vartheta (t)}{t} {\: \rm d} t\Big)
	\end{align*}
	where  $\vartheta \in \calC^2([x_0, \infty))$ is positive decreasing real function satisfying
	\begin{align*}
		\lim_{x\to\infty} \vartheta(x)=0,\qquad
		\lim_{x\to\infty} \frac{x\vartheta'(x)}{\vartheta(x)}=0, \qquad
		\lim_{x\to\infty}\frac{x^2\vartheta''(x)}{\vartheta(x)}=0,
	\end{align*}
	and for every $\varepsilon>0$ there is a constant $C_{\varepsilon}>0$ such that
	$1\le C_{\varepsilon}\vartheta(x)x^{\varepsilon}$ and $\lim_{x\to\infty}\ell(x)=\infty$.
\end{definition}
Finally, we define the subfamily $\calR_c$ of regularly varying functions.
\begin{definition}
	\label{def:3}
	For every $c \in(0, 2) \setminus \{1\}$ let $\mathcal R_c$ be a family of
	increasing, convex, regularly-varying functions $h:[x_0, \infty) \rightarrow [1, \infty)$ of the form
	\[
		h(x)=x^c L(x)
	\]
	where $L \in \mathcal L$. If $c=1$ we impose that $L \in \mathcal L_0$.
\end{definition}
We fix two functions $h_1 \in \calR_{c_1}$ and $h_2 \in \calR_{c_2}$ for $c_1 \in [1, 2)$ and $c_2 \in [1, 6/5)$.
Let $\varphi_1$ and $\varphi_2$ be the inverse of $h_1$ and $h_2$, respectively. We consider
a function $\psi: [x_0, \infty) \rightarrow (0, \infty)$ such that for all $x \geq x_0$,
$\psi(x) \leq 1/2$ and
\begin{align}
	\label{eq:5}
	\lim_{x \to +\infty} \frac{\psi(x)}{\varphi_2'(x)} = 1, \quad
	\lim_{x \to +\infty} \frac{\psi'(x)}{\varphi_2''(x)} = 1, \quad
	\lim_{x \to +\infty} \frac{\psi''(x)}{\varphi_2'''(x)} = 1.
\end{align}
Finally, we define two sets
\[
	\mathbf B_+=\big\{n\in\NN: \{\varphi_1(n)\}<\psi(n) \big\},
	\quad
	\mathbf B_-=\big\{n\in\NN: \{-\varphi_1(n)\}<\psi(n) \big\}.
\]
Let us observe that if $h_1=h_2=h$ is the inverse function $\varphi$ and
$\psi(x)=\varphi(x+1)-\varphi(x)$ then $\mathbf B_-=\mathbf A$. Indeed, we have the following
chain of equivalences
\begin{align*}
	m \in \mathbf A &\iff m = \lfloor h(n) \rfloor \ \text{ for some } n\in\NN \\
	& \iff h(n)-1 < m \leq h(n) < m+1 \\
	& \iff \varphi(m) \leq n < \varphi(m+1), \; \; \text{ since $\varphi$ is well-defined and monotonically increasing} \\
	& \iff 0 \leq n- \varphi(m) < \varphi(m+1) - \varphi(m) = \psi(m) < 1/2 \\
	& \iff 0 \leq \{ - \varphi(m) \} < \psi(m) \\
	& \iff m \in \mathbf B_-.
\end{align*}
The main result of this paper is the following theorem.
\begin{theorem}
	\label{thm:1}
	Assume that $c_1\in[1, 2)$ and $c_2=1$. Then for every $p\ge 2$ there exists a constant
	$\mathbf C_p>0$ such that for every $N\in\NN$ and any sequence of complex numbers
	$(a_n: n\in\NN)$ satisfying $\sup_{n\in\NN}|a_n|\le1$ we have
	\begin{align}
		\label{eq:30}
		\Big\|
		\sum_{n\in \mathbf B_N^{\pm}}a_ne^{2\pi i n\xi}
		\Big\|_{L^p(\TT, {\rm d}\xi)}
		\le
		\mathbf C_p
		\Big\|
		\sum_{n\in \mathbf B_N^{\pm}}e^{2\pi i n\xi}
		\Big\|_{L^p(\TT, {\rm d}\xi)}
	\end{align}
	where $\mathbf B_N^{\pm}=\mathbf B_{\pm}\cap[1, N]$.
\end{theorem}
We observe that by the Hausdorff--Young inequality for every $p \geq 2$ we obtain
\[
	\Big\|
	\sum_{n \in \mathbf B_N^{\pm} }a_ne^{2\pi i n\xi}
	\Big\|_{L^p(\TT, {\rm d}\xi)}
	\le
	|\mathbf B_N^\pm|^{1/p'}.
\]
Moreover, by integrating over frequencies $|\xi|\leq 1/(100N)$, we have the following lower bound
\[
	\Big\|
	\sum_{n\in \mathbf B_N^\pm}e^{2\pi i n\xi}
	\Big\|_{L^p(\TT, {\rm d}\xi)}
	\gtrsim
	|\mathbf B_N^\pm| N^{-1/p}.
\]
These inequalities combined together yield
\begin{align}
	\label{eq:29}
	\Big\|
	\sum_{n\in \mathbf B_N^\pm }a_ne^{2\pi i n\xi}
	\Big\|_{L^p(\TT, {\rm d}\xi)}
	\lesssim
	|\mathbf B_N^\pm|^{1/p} N^{-1/p}
	\Big\|
	\sum_{n\in \mathbf B_N^\pm}e^{2\pi i n\xi}
	\Big\|_{L^p(\TT, {\rm d}\xi)}.
\end{align}
By Proposition \ref{prop:1} for $c_2 = 1$, we have $|\mathbf B_N^{\pm}| \sim \varphi_2(N)$
where $\varphi_2(N)=N L_{\varphi_2}(N)$ for some slowly varying function $L_{\varphi_2} \in \calL_0$.
Therefore, applying  inequality \eqref{eq:29}, we obtain
\[
	\mathbf C_p(\mathbf B_N^{\pm}, N) \lesssim L_{\varphi_2}(N)^{1/p} \lesssim N^{\varepsilon}
\]
for any $\varepsilon>0$. Hence, the main difficulty in proving Theorem \ref{thm:1} is to show that the constant
in \eqref{eq:30} is independent of $N$.

Next, we would like to relax the hypothesis in Theorem \ref{thm:1} to allow any $c_2\in[1, 6/5)$.
It is possible at the expense of a slightly worse range of $p$. Let us introduce
\marginpar{check}
\[
	p(c_1, c_2)=\frac{2/c_1-6/c_2+6}{1/c_1+3/c_2-3}=2+\frac{12-12/c_2}{1/c_1+3/c_2 - 3}.
\]
We observe that if $c_1 \in [1, 2)$ and $c_2 \in[1, 6/5)$ then $1<\frac{1}{3c_1}+\frac{1}{c_2}$, thus
\[
	\frac{12-12/c_2}{1/c_1+3/c_2 - 3} \ge 0.
\]
Also notice that
\[
	\lim_{c_2 \to 1} p(c_1, c_2)=2.
\]
The extended version of Theorem \ref{thm:1} has the following form.
\begin{theorem}
  \label{thm:2}
	Assume that $c_1 \in [1, 2)$ and $c_2 \in [1, 6/5)$. Then for every $p \ge p(c_1, c_2)$ there
	exists a constant $\mathbf C_p>0$ such that for every $N\in\NN$ and any sequence of complex
	numbers $(a_n: n\in\NN)$ satisfying $\sup_{n\in\NN}|a_n|\le1$ we have
	\[
		\Big\|
		\sum_{n\in \mathbf B_N^{\pm}}a_ne^{2\pi i n\xi}
		\Big\|_{L^p(\TT, {\rm d}\xi)}
		\le
		\mathbf C_p
		\Big\|
		\sum_{n\in \mathbf B_N^{\pm}}e^{2\pi i n\xi}
		\Big\|_{L^p(\TT, {\rm d}\xi)}
	\]
	where $\mathbf B_N^{\pm}=\mathbf B_{\pm}\cap[1, N]$.
\end{theorem}

We were inspired to study Hardy--Littlewood majorant property 
 by the paper of Mockenhaupt and Schlag \cite{MoSc} where the authors
 considered sparse random subsets of the integers. The desire to
 better understand structure of deterministic sets which satisfy
 \eqref{eq:0} was our principal motivation.   

Before turning to the arguments, let us begin with some preliminary remarks.
The heart of the matter lies in proving our Proposition \ref{prop:2},
which can be though of as a restriction estimate for our sets
$\mathbf{B}_N^{\pm}$. We accomplish this using a Tomas--Stein $TT^*$
argument, which forces us to estimate certain exponential sums, see
Section 3 below. These estimates are quite delicate, and lead to the
technical restriction on the range of $L^p$ spaces which we are able
to handle; in particular, we do not yet know how to extend Theorem 2
to the full regime $2< p < p(c_1,c_2)$. Finally, it is worth calling
attention to the explicit construction of the sets
$\mathbf{B}_N^{\pm}$ for which the full strength of the
Hardy--Littlewood property holds. To the best of the authors knowledge it is
the first treatment where such a wide family of subsets of the integers
satisfies property \eqref{eq:0}.

\section{Some properties of the sets $\mathbf B_{\pm}$}
As it has been observed, when $c_1\in[1, 2)$ and $c_2\in[1, 6/5)$ we have $1<\frac{1}{3c_1}+\frac{1}{c_2}$, or equivalently 
\[ 3(1-\gamma_2)+(1-\gamma_1)<1,\]
where $\gamma_1=1/c_1$ and $\gamma_2 = 1/c_2$.
Under this assumption, we prove the asymptotic formula for the cardinality of sets $\mathbf B_{\pm}$.
\begin{proposition}
	\label{prop:1}
	For every $\epsilon > 0$
	\begin{align}
		\label{eq:9}
		|\mathbf B_N^\pm| = \varphi_2(N) \big(1 + \calO\big(N^{-\epsilon}\big)\big).
	\end{align}
\end{proposition}
From now on we only work with the sets $\mathbf B_+$ because all the results remain
valid for $\mathbf B_-$ with similar proofs. To simplify the notation we write
\[
	\mathbf B=\mathbf B_+=\{n\in\NN: \{\varphi_1(n)\}<\psi(n)\}.
\]
We need the following working characterizations of the sets $\mathbf B$.
\begin{lemma}
	\label{equiv}
	$n\in\mathbf B$ if and only if
	$\lfloor \varphi_1(n) \rfloor - \lfloor \varphi_1(n) - \psi(n) \rfloor = 1$.
\end{lemma}
\begin{proof}
	We begin with the forward implication; it suffices to show that if $n \in \mathbf B$,
	the integer
	\[
		\lfloor \varphi_1(n) \rfloor - \lfloor \varphi_1(n) - \psi(n) \rfloor
	\]
	belongs to $(0,3/2)$. By definition, if
	$n \in \mathbf B$ then $0\leq \varphi_1(n) - \lfloor \varphi_1(n) \rfloor < \psi(n)$, thus
	\[
		- \varphi_1(n) \leq - \lfloor  \varphi_1(n) \rfloor < \psi(n) - \varphi_1(n)
	\]
	if and only if
	\[
		\varphi_1(n) \geq \lfloor  \varphi_1(n) \rfloor > \varphi_1(n) - \psi(n),
	\]
	from where it follows that
	\[
		\lfloor \varphi_1(n)\rfloor - \lfloor \varphi_1(n) - \psi(n) \rfloor
		> \{ \varphi_1(n) - \psi(n) \} \geq 0.
	\]
	In view of $\lfloor  \varphi_1(n) - \psi(n) \rfloor \geq \varphi_1(n) - \psi(n) -1$, we obtain
	\begin{align*}
		\lfloor \varphi_1(n) \rfloor - \lfloor \varphi_1(n) - \psi(n) \rfloor
		& \leq \lfloor  \varphi_1(n) \rfloor - \varphi_1(n) + \psi(n) + 1\\
		& \leq \psi(n) +1 < 3/2.
	\end{align*}
	We now turn to the reverse implication; if
	$\lfloor \varphi_1(n) \rfloor = 1 +\lfloor  \varphi_1(n) - \psi(n) \rfloor$, we have
	\begin{align*}
		0 & \leq \varphi_1(n) - \lfloor  \varphi_1(n) \rfloor
		= \varphi_1(n) -1 - \lfloor  \varphi_1(n) - \psi(n) \rfloor \\
		& < \varphi_1(n) - 1 + 1 + \psi(n) - \varphi_1(n) = \psi(n).
	\end{align*}
	Consequently, we get $\{ \varphi_1(n) \} < \psi(n)$, as desired.
\end{proof}

Our next task is to show that for every $\delta\ge0$ satisfying $3(1-\gamma_2)+(1-\gamma_1)+6\delta<1$
there is $\delta'>0$ such that
\begin{align}
	\label{eq:20}
	\sum_{n\in\mathbf{B}_N}e^{2\pi i \xi n}
	=\sum_{n=1}^N\psi(n)e^{2\pi i \xi n}+\mathcal O\big(\varphi_2(N)N^{-\delta-\delta'}\big)
\end{align}
where the implied constant is independent of $\xi$ and $N$. Let us observe that the asymptotic formula \eqref{eq:9}
follows from \eqref{eq:20} by taking $\xi = 0$. Indeed, we have
\begin{align*}
	|\mathbf B_N|
	=
	\sum_{n\in\mathbf{B}_N}1=\sum_{n=1}^N\psi(n)+\mathcal O\big(\varphi_2(N)N^{-\varepsilon}\big)
\end{align*}
and summation by parts yields
\begin{align}
	\label{eq:32}
	\frac{1}{\varphi_2(N)}\sum_{n=1}^N\psi(n)
	=\frac{N\psi(N)}{\varphi_2(N)} - \frac{1}{\varphi_2(N)}\int_1^Nx\psi'(x) {\: \rm d}x
	=\frac{1}{\varphi_2(N)}\int_1^N\psi(x)dx
	=1 + o(1).
\end{align}
Although, for the proof of \eqref{eq:9} we only needed \eqref{eq:20} with $\xi=0$, the more general version
will be used in our future works.

For the proof of \eqref{eq:20}, let us introduce the ``sawtooth'' function $\Phi(x)=\{x\}-1/2$. Notice that
\begin{align*}
	\lfloor \varphi_1(n) \rfloor - \lfloor \varphi_1(n) - \psi(n) \rfloor
	=\psi(n)+\Phi\big(\varphi_1(n)-\psi(n) \big)-\Phi\big(\varphi_1(n)\big).
\end{align*}
With this in mind, we may write
\begin{align}
	\label{eq:13}
	\sum_{n\in\mathbf{B}_N}e^{2\pi i \xi n}
	=\sum_{n=1}^N\psi(n)e^{2\pi i \xi n}
	+\sum_{n=1}^N\big(\Phi\big(\varphi_1(n)-\psi(n) \big)-\Phi\big(\varphi_1(n)\big)\big)e^{2\pi i \xi n}.
\end{align}
The second sum we absorb into an error term of the order $\calO\big(\varphi_2(N) N^{-\varepsilon}\big)$. To do so,
see \cite{hb}, we expand $\Phi$ into its Fourier series, i.e.
\begin{align*}
	\Phi(x)=\sum_{0<|m|\le M} \frac{1}{2\pi i m} e^{-2\pi imx}
	+\calO\left(\min\left\{1, \frac{1}{M\|x\|}\right\}\right),
\end{align*}
for some $M>0$ where $\|x\| = \min\{\norm{x-n} : n \in \ZZ\}$ is the distance of $x \in \RR$ to the nearest
integer. Next, we expand
\begin{align}
	\label{eq:15}
	\min\left\{1, \frac{1}{M\|x\|}\right\}=\sum_{m\in\ZZ}b_m e^{2\pi imx}
\end{align}
where
\begin{align}
	\label{eq:16}
	|b_m|\lesssim \min\left\{\frac{\log M}{M}, \frac{1}{|m|}, \frac{M}{|m|^2}\right\}.
\end{align}
We split the second sum in \eqref{eq:13} into three parts,
\begin{align*}
	I_1&=\sum_{0<|m|\le M}\frac{1}{2\pi i m}
	\sum_{n=1}^Ne^{2\pi i(n\xi-m\varphi_1(n))}\big(e^{2\pi im\psi(n)} - 1 \big),\\
	I_2&=\mathcal O\bigg( \sum_{n=1}^N\min\left\{1, \frac{1}{M\|\varphi_1(n)-\psi(n)\|}\right\}\bigg),\\
	I_3&=\mathcal O\bigg( \sum_{n=1}^N\min\left\{1, \frac{1}{M\|\varphi_1(n)\|}\right\}\bigg).
\end{align*}
Now, our aim is to show that each part $I_1, I_2$ and $I_3$ is $\calO\big(\varphi_2(N)N^{-\varepsilon}\big)$. In the proof
we use the estimates for the following trigonometric sums: for $m \in \ZZ \setminus \{0\}$, $l \in \{0, 1\}$ and
$X \leq X' \leq 2X$ we consider
\[
	\sum_{X \leq n \leq X' \leq 2X}
 	e^{2\pi i(\xi n + m(\varphi_1(n)-l\psi(n)))}
\]
By \cite[Lemma 2.14]{m1}, if $c_1 = 1$ then there is a positive decreasing real function $\sigma_1$ satisfying
$\sigma_1(2x) \simeq \sigma_1(x)$ and $\sigma_1(x) \gtrsim x^{-\varepsilon}$ for any $\varepsilon > 0$, such that
\begin{equation}
	\label{eq:1}
	\varphi_1''(x) \simeq \frac{\varphi_1(x) \sigma_1(x)}{x^2}.
\end{equation}
We set $\sigma_1 \equiv 1$ whenever $c_1 > 1$. Similarly, by \cite[Lemma 2.14]{m1} for $\varphi_2'''$ we obtain
\begin{equation}
	\label{eq:2}
	\varphi_2'''(x) \simeq \frac{\varphi_2(x)}{x^3}.
\end{equation}
Therefore, by \eqref{eq:5} we may write
\[
	\abs{\psi''(x)} \simeq \abs{\varphi_2'''(x)} \simeq \frac{\varphi_2(x)}{x^3}.
\]
Since $1/c_2 \leq 1 < 1+1/c_1$, we get
\[
	\frac{\varphi_2(x)}{x \sigma_1(x) \varphi_1(x)} = o(1),
\]
thus
\begin{equation*}
	\abs{\psi''(x)} = o\Big(\frac{\sigma_1(x) \varphi_1(x)}{x^2}\Big).
\end{equation*}
Let $F(x)=\xi x + m(\varphi_1(x)-l\psi(x))$. By \eqref{eq:1} and \eqref{eq:2}, for any $X \leq x \leq X' \leq 2X$
we may write
\[
	\abs{F''(x)} = \abs{m} \cdot \abs{\varphi_1''(x) - l \psi''(x)}
	\simeq \frac{\abs{m} \sigma_1(X) \varphi_1(X)}{X^2}.
\]
Therefore, the Van der Corput lemma (see \cite[Theorem 2.2]{gk}) yields
\begin{align}
	\nonumber
	\Big|\sum_{X < n \le 2X}
	e^{2\pi i(\xi n + m(\varphi_1(n)-l\psi(n)))}
	\Big|
	&\lesssim
	X \left(\frac{m\sigma_1(X)\varphi_1(X)}{X^2}\right)^{1/2}
	+\left(\frac{X^2}{m\sigma_1(X)\varphi_1(X)}\right)^{1/2}\\
	\label{eq:22}
	&\lesssim m^{1/2}X\big(\sigma_1(X)\varphi_1(X)\big)^{-1/2}.
\end{align}
Finally, we get
\begin{align*}
	\Big|
	\sum_{j = 1}^N e^{2\pi i(\xi n + m(\varphi_1(n)-l\psi(n)))}
	\Big|
	&\leq
	\sum_{j = 0}^{\lceil \log N \rceil}
	\Big|
	\sum_{\atop{2^j < n \leq 2^{j+1}}{n \leq N}}
	e^{2\pi i (\xi n + m(\varphi_1(n) - l \psi(n)))}
	\Big|\\
	&\lesssim
	m^{1/2} N (\log N) \big(\sigma_1(N) \varphi_1(N)\big)^{-1/2}
\end{align*}
since the function $x \mapsto x \big(\sigma_1(x) \varphi_1(x) \big)^{-1/2}$ is increasing.
In particular, we have proven the following lemma.
\begin{lemma}
	\label{lem:1}
	There is a positive decreasing real function $\sigma_1$ satisfying $\sigma_1(2x) \simeq \sigma_1(x)$
	and $\sigma_1(x) \gtrsim x^{-\varepsilon}$, for any $\varepsilon > 0$, such that for every
	$m \in \ZZ \setminus \{0\}$, $l\in\{0, 1\}$, and $N\ge 1$ we have
	\begin{align*}
		\Big|\sum_{1\le n\le N}\ e^{2\pi i(\xi n + m(\varphi_1(n)-l\psi(n)))}\Big|
		\lesssim
		|m|^{1/2} N (\log N) \big(\sigma_1(N)\varphi_1(N)\big)^{-1/2}.
	\end{align*}
	If $c_1>1$ then $\sigma_1\equiv1$. The implied constant is independent of $m$, $N$ and $\xi$.
\end{lemma}

Next, we return to bounding $I_1$, $I_2$ and $I_3$.
\subsection{The estimate for $I_1$}
\label{sub:2.1}
Let
\[
	S(x)=\sum_{x\le n\le x'< 2x}e^{2\pi i(n\xi-m\varphi_1(n))}
\]
and $\phi_m(x)=e^{2\pi im\psi(x)} -1 $. We observe that
\[
	|\phi_m(x)| \lesssim m x^{-1} \varphi_2(x)
\]
and
\[
	|\phi_m'(x)|\lesssim mx^{-2} \varphi_2(x).
\]
Applying to the inner sum in $I_1$ summation by parts together with \eqref{eq:22} we obtain
\begin{multline*}
	\Big|
	\sum_{n=1}^Ne^{2\pi i(n\xi-m\varphi_1(n))}\phi_m(n)
	\Big| \\
	\le
	(\log N)
	\sup_{X \in [1, N]}
	\Big(|S(2X)|\cdot |\phi_m(2X)|+|S(X)|\cdot|\phi_m(X)|+\int_X^{2X} |S(x)| \cdot |\phi_m'(x)| {\: \rm d}x\Big)\\
	\lesssim
	(\log N) \sup_{X\in[1, N]} m^{3/2} \varphi_2(X) \big(\sigma_1(X) \varphi_1(X)\big)^{-1/2} \\
	\le
	m^{3/2}\varphi_2(N) (\log N) \big(\sigma_1(N)\varphi_1(N)\big)^{-1/2}.
\end{multline*}
Therefore,
\begin{align*}
	\abs{I_1}
	\lesssim
	\sum_{m=1}^M m^{1/2} (\log N) \varphi_2(N)\big(\sigma_1(N)\varphi_1(N)\big)^{-1/2} \\
	\lesssim
	M^{3/2} (\log N) \varphi_2(N) \big(\sigma_1(N)\varphi_1(N)\big)^{-1/2}.
\end{align*}

\subsection{The estimates for $I_2$ and $I_3$}
\label{sub:2.2}
We only treat $I_2$ because $I_3$ can be handled by a similar reasoning. By \eqref{eq:15}, \eqref{eq:16} and Lemma
\ref{lem:1} we have
\begin{multline*}
	\sum_{n=1}^N\min\left\{1,  \frac{1}{M\|\varphi_1(n)-\psi(n)\|}\right\}
	\le
	\sum_{m\in\ZZ}
	|b_m|
	\bigg|\sum_{n=1}^N e^{2\pi i m(\varphi_1(n)-\psi(n))} \bigg|\\
	\lesssim
	\frac{N (\log M)}{M}+\bigg(\sum_{0<|m| \le M}\frac{\log M}{M}+\sum_{|m|>M}\frac{M}{|m|^2}\bigg)
	|m|^{1/2} (\log N) \varphi_2(N)\big(\sigma_1(N)\varphi_1(N)\big)^{-1/2} \\
	\lesssim \frac{N (\log M)}{M}+M^{1/2} (\log M) (\log N) \varphi_2(N)\big(\sigma_1(N)\varphi_1(N)\big)^{-1/2}.
\end{multline*}

\subsection{Concluding remarks}
Based on Subsections \ref{sub:2.1} and \ref{sub:2.2} we get
\begin{align*}
	|I_1|+|I_2|+|I_3|
	\lesssim
	\frac{N\log M}{M}+M^{3/2} (\log M) (\log N) \varphi_2(N)\big(\sigma_1(N)\varphi_1(N)\big)^{-1/2}.
\end{align*}
Therefore, by taking $M=N^{1+\delta} (\log N) \varphi_2(N)^{-1}$, we conclude
\begin{align*}
	\abs{I_1} + \abs{I_2} + \abs{I_3}
	& \lesssim
	\varphi_2(N) N^{-\delta} \Big(1+N^{3/2+5\delta/2} (\log N) \sigma_1(N)^{-1/2}\varphi_1(N)^{-1/2}
	\varphi_2(N)^{-3/2}\Big)\\
	& \lesssim
	\varphi_2(N)N^{-\delta} \Big(1+N^{3/2+6\delta/2-\gamma_1/2-3\gamma_2/2}\Big)
\end{align*}
which is bounded by a constant multiple of $\varphi_2(N) N^{-\delta}$ since $3(1-\gamma_2)+(1-\gamma_1)+6\delta<1$.

\section{Proof of Theorem \ref{thm:1}}
Let
\[
	\calF(f)(\xi)=\sum_{n\in\ZZ}f(n)e^{2\pi i \xi n}
\]
denote the Fourier transform on $\ZZ$, and
\[
	\hat{f}(n)=\int_{\TT}f(\xi)e^{-2\pi i \xi n} {\: \rm d} \xi
\]
denote the Fourier transform on $\TT$. For any measure space $X$, let $\calC(X)$ be the space of all continuous
functions on $X$. For $N\in\NN$ we introduce on $\ZZ$ a measure
$\mu_N$ defined
\[
	\mu_N(x)=N^{-1}\sum_{n\in\mathbf B_N}\psi(n)^{-1}\delta_n(x).
\]
Let $T_N: \calC(\mathbf B_N) \rightarrow \calC(\TT)$ be the linear operator given by
\begin{align*}
	T_N(f)=\calF\big(f\mu_N\big).
\end{align*}
We are going to prove the following proposition.
\begin{proposition}
	\label{prop:2}
	For each
	\begin{equation}
		\label{eq:24}
		p \geq 2 + \frac{12 - 12/c_2}{1/c_1 + 3/c_2 - 3}
	\end{equation}
	there is a constant $C_p > 0$ such that for all $N \in \NN$ and
	$f \in L^2\big(\mathbf B_N, \mu_N\big)$
	\[
		\lVert
			T_N f
		\rVert_{L^p(\TT)}
		\leq
		C_p N^{-1/p} \lVert f \rVert_{L^2(\mathbf B_N, \mu_N)}.
	\]
\end{proposition}

Before embarking on the proof we show the following.
\begin{lemma}
	\label{lem:2}
		For every $\delta > 0$ satisfying $(1-\gamma_1)+3(1-\gamma_2) + 6 \delta < 1$ there is $\delta' > 0$
		such that
		\begin{align*}
			\sum_{n \in \mathbf B_N} \psi(n)^{-1} e^{2\pi i\xi n}
			=
			\sum_{n = 1}^N e^{2\pi i\xi n}
			+
			\calO\big(N^{1 - \delta - \delta'}\big).
		\end{align*}
		The implied constant is independent of $\xi$ and $N$.
\end{lemma}
\begin{proof}
	For $N \in \NN$ and $\xi \in \TT$ we set
	\[
		S_N(\xi) = \sum_{k \in \mathbf B_N} e^{2\pi i \xi k}.
	\]
	Then, by the summation by parts we have
	\begin{align}
		\nonumber
		\sum_{n \in \mathbf B_N} \psi(n)^{-1} e^{2\pi i \xi n}
		&=
		\sum_{n = 1}^N \psi(n)^{-1} \big(S_n(\xi) - S_{n-1}(\xi)\big)\\
		\label{eq:14}
		&=
		\psi(N+1)^{-1} S_N(\xi)
		+
		\sum_{n = 1}^N \big(\psi(n)^{-1} - \psi(n+1)^{-1}\big) S_n(\xi).
	\end{align}
	Similarly, we may write
	\begin{equation}
		\label{eq:23}
		\sum_{n = 1}^N e^{2\pi i \xi n}
		=
		\psi(N+1)^{-1} \sum_{n = 1}^N \psi(n) e^{2\pi i \xi n}
		+ \sum_{n = 1}^N \big(\psi(n)^{-1} - \psi(n+1)^{-1}\big) \sum_{k = 1}^n \psi(k) e^{2\pi i \xi k} .
	\end{equation}
	Thus, subtracting \eqref{eq:23} from \eqref{eq:14} we may estimate
	\begin{multline*}
		\Big\lvert
		\sum_{n \in \mathbf B_N} \psi(n)^{-1} e^{2\pi i \xi n}
		-
		\sum_{n = 1}^N e^{2\pi i \xi n}
		\Big\rvert
		\leq
		\sum_{n = 1}^N \big\lvert \psi(n)^{-1} - \psi(n+1)^{-1} \big\rvert
		\cdot
		\Big\lvert
		S_n(\xi) - \sum_{k=1}^n \psi(k) e^{2\pi i \xi k}
		\Big\rvert\\
		+
		\psi(N+1)^{-1}
		\Big\lvert
		S_N(\xi) - \sum_{k=1}^N \psi(k) e^{2\pi i \xi k}
		\Big\rvert.
	\end{multline*}
	By \eqref{eq:20}, for any $\delta > 0$ satisfying $3(1-\gamma_1) + (1 - \gamma_2) + 6 \delta < 1$ there
	is $\delta' > 0$  such that for all $n \in \NN$
	\[
		\Big\lvert
		S_n(\xi) - \sum_{k = 1}^n \psi(k) e^{2 \pi i \xi k}
		\Big\rvert
		\leq
		C \varphi_2(n) n^{-\delta - \delta'}.
	\]
	Using \eqref{eq:5} together with \cite[Lemma 2.14]{m1} we obtain
	\[
		\psi'(n) \lesssim \varphi_2''(n) \lesssim \frac{\varphi_2(n)}{n^2}.
	\]
	Therefore, again by \eqref{eq:5} and the monotonicity of $\varphi_2$ we get
	\[
		\big\lvert
		\psi(n)^{-1} - \psi(n+1)^{-1}
		\big\rvert
		\lesssim
		\sup_{t \in [n, n +1]}
		\big\lvert
		\psi(t)^{-2} \psi'(t)
		\big\rvert
		\lesssim
		\varphi_2(n)^{-1}.
	\]
	Hence,
	\[
		\sum_{n = 1}^N \big\lvert \psi(n)^{-1} - \psi(n+1)^{-1}\big\rvert
		\cdot
		\Big\lvert
		S_n(\xi) - \sum_{k=1}^n \psi(k) e^{2\pi i \xi k}
		\Big\rvert
		\lesssim
		\sum_{n = 1}^N n^{-\delta - \delta'}
		\lesssim
		N^{1- \delta- \delta'}. \qedhere
	\]
\end{proof}

\begin{proof} [Proof of Proposition \ref{prop:2}]
  The $TT^*$ argument will be critical in the proof. Firstly, let us calculate $T_N^*$. By Plancherel's theorem we have
	\begin{align*}
		\langle T_N f, g\rangle_{L^2(\TT)}
		= \int_{\TT}\calF\big(f\mu_N\big)(\xi) \overline{g(\xi)} {\: \rm d}\xi
		= \sum_{n\in\ZZ} f(n)\overline{\hat{g}(n)}\mu_N(n)
		= \langle f, T_N^*g \rangle_{L^2(\mathbf B_N, \mu_N)}.
	\end{align*}
	Therefore, the adjoint operator $T_N^*: \calC(\TT)^* \rightarrow \calC(\mathbf B_N)^*
	=\calC(\mathbf B_N)$ is given by
	\begin{align*}
		T_N^*(g)=\hat{g} \cdot \ind{\mathbf B_N},
	\end{align*}
	and consequently, $T_N T_N^*: \calC(\TT)^* \rightarrow \calC(\TT)^*$ may be written as
	\begin{align*}
		T_N T_N^*f=f*\calF(\mu_N).
	\end{align*}
	Let us observe that it is enough to show
	\begin{align}
		\label{eq:12}
		\lVert T_N T_N^* \rVert_{L^{p'}(\TT) \rightarrow L^p(\TT)} \leq C_p N^{-2/p}.
	\end{align}
	Indeed, for $f \in L^2\big(\mathbf B_N, \mu_N\big)$ and $g \in L^{p'}(\TT)$ we have
	\[
		|\langle T_N f, g \rangle_{L^2(\TT)}|
		=
		|\langle f, T_N^* g \rangle_{L^2(\mathbf B_N, \mu_N)}|
		\leq
		\lVert f \rVert_{L^2(\mathbf B_N, \mu_N)}
		\lVert T_N^* g \rVert_{L^2(\mathbf B_N, \mu_N)}
	\]
	and since
	\[
		\lVert T_N^* g \rVert_{L^2(\mathbf B_N, \mu_N)}^2
		=
		\langle T_N T_N^* g, g \rangle_{L^2(\TT)}
		\leq
		\lVert T_N T_N^* \rVert_{L^{p'}(\TT) \rightarrow L^p(\TT)}
		\lVert g \rVert_{L^{p'}(\TT)}^2,
	\]
	we obtain
	\[
		\lVert T_N f \rVert_{L^p(\TT)}
		\leq
		\lVert T_N T_N^* \rVert_{L^{p'}(\TT) \rightarrow L^p(\TT)}^{1/2}
		\lVert f \rVert_{L^2(\mathbf B_N, \mu_N)}.
	\]
	For the proof of \eqref{eq:12},  for $N\in\NN$, let us introduce an auxiliary measure $\nu_N$ on $\ZZ$ and the corresponding linear operator
	$S_N: \calC(\NN_N) \rightarrow \calC(\TT)$, by setting
	\[
		\nu_N(x) = N^{-1}\sum_{n \in \NN_N} \delta_n(x),
	\]
	and
	\[
		S_N f = \calF(f \nu_N).
	\]
Here, $\mathbb{N}_N := \mathbb{N} \cap [1,N]$.
	Reasoning similar to the above applied to the operator $S_N$ leads to
	\begin{align*}
		S_N S_N^*f=f*\calF(\nu_N).
	\end{align*}
	Since $L^p(\TT)$ can be embedded into $\calC(\TT)^*$ for any $p \ge 1$ we may consider the operators
	$T_N T_N^*$ and $S_N S_N^*$ as mappings on $L^p(\TT)$ spaces.
	Next, we write
 	\begin{align*}
   		\big\lVert
		T_N T_N^*f
		\big\rVert_{L^{p}(\TT)}&
		=
		\big\lVert
		f*\calF(\mu_N)
		\big\rVert_{L^{p}(\TT)}\\
		& \leq
		\big\lVert
		f * \calF(\nu_N)
		\big\rVert_{L^{p}(\TT)}
		+
   		\big\lVert
		f*\calF(\mu_N-\nu_N)
		\big\rVert_{L^{p}(\TT)}.
	\end{align*}
	We are going to show that for each $p$ satisfying \eqref{eq:24} there is $C_p > 0$ such that
	\begin{align}
		\label{eq:26}
		\big\lVert
		f * \calF(\nu_N)
		\big\rVert_{L^{p}(\TT)}
		&\leq
		C_p
		N^{-2/p}
		\lVert f \rVert_{L^{p'}(\TT)},\\
		\label{eq:33}
		\big\lVert
		f * \calF(\mu_N - \nu_N)
		\big\rVert_{L^p(\TT)}
		&\leq
		C_p
		N^{-2/p}
		\lVert f \rVert_{L^{p'}(\TT)}
	\end{align}
	for all $f \in L^{p'}(\TT)$. We start by proving \eqref{eq:26} for $p = 2$. By Plancherel's theorem we have
	\begin{align*}
		\big\lVert
		f *\calF(\nu_N)
		\big\rVert_{L^2(\TT)}
		=
		\big\lVert
		\hat{f} \nu_N
		\big\rVert_{\ell^2(\ZZ)}
		& \leq
		\lVert \nu_N \rVert_{\ell^\infty(\ZZ)} \lVert f \rVert_{L^2(\TT)} \\
		& \leq
		N^{-1}
		\lVert f \rVert_{L^2(\TT)}.
	\end{align*}
	On the other hand, for $p = \infty$ we may write
	\begin{align*}
		\lVert
		f * \calF(\nu_N)
		\rVert_{L^\infty(\TT)}
		\leq
		\lVert \calF(\nu_N) \rVert_{L^\infty(\TT)}
		\lVert f \rVert_{L^1(\TT)}		
\leq
		\lVert f \rVert_{L^1(\TT)}.
	\end{align*}	

Therefore, for $p \geq 2$ we use  Riesz--Thorin interpolation theorem to obtain \eqref{eq:26}.
	To show \eqref{eq:33}, we apply analogous reasoning. Firstly, by Plancherel's theorem we have
	\begin{align*}
		\big\lVert
		f * \calF(\mu_N - \nu_N)
		\big\rVert_{L^2(\TT)}
		=
		\big\lVert
		\hat{f}
		\big(\mu_N - \nu_N\big)
		\big\rVert_{L^2(\TT)}
		& \leq
		\big\lVert
		\mu_N - \nu_N
		\big\rVert_{\ell^\infty(\ZZ)}
		\lVert f \rVert_{L^2(\TT)}\\
		& \leq
		\varphi_2(N)^{-1} \lVert f \rVert_{L^2(\TT)}.
	\end{align*}
	Secondly, for $p = \infty$ we get
	\begin{align*}
		\lVert
		f * \calF(\mu_N - \nu_N)
		\rVert_{L^\infty(\TT)}
		& \leq
		\lVert \calF(\mu_N) - \calF(\nu_N) \rVert_{L^\infty(\TT)}
		\lVert f \rVert_{L^1(\TT)}\\
		& \leq
		N^{-\delta - \delta'}
		\lVert f \rVert_{L^1(\TT)}
	\end{align*}
	where in the last estimate we have used Lemma \ref{lem:2}. Thus, again by Riesz--Thorin interpolation theorem,
	for $p \geq 2$ we get
	\begin{align*}
		\big\lVert
		f*\calF(\mu_N-\nu_N)
		\big\rVert_{L^{p}(\TT)}
	 	& \leq
		\big\lVert \mu_N-\nu_N \big\rVert_{\ell^{\infty}(\ZZ)}^{2/p}
		\cdot
		\big\lVert \calF(\mu_N-\nu_N) \big\rVert_{L^{\infty}(\TT)}^{1-2/p}
		\cdot
		\|f\|_{L^{p'}(\TT)}\\
		&\lesssim
		\varphi_2(N)^{-2/p} N^{-(\delta+\delta')(1-2/p)}
		\|f\|_{L^{p'}(\TT)}.
	\end{align*}
	Let us recall that for any $\varepsilon > 0$
	\[
		\varphi_2(N) \gtrsim_\varepsilon N^{\gamma_2 - \varepsilon}
	\]
	Therefore, for the inequality \eqref{eq:33} to hold true, we
        need to have $\varepsilon>0$ and $p>2$ to satisfy
	\[
		-2 (\gamma_2 - \varepsilon)/p - (\delta + \delta')(1-2/p) \leq -2/p.
	\]
	Hence,
	\[
		\varepsilon \leq -(1-\gamma_2) + (\delta + \delta')(p-2)/2.
	\]
	Because the right hand side has to be positive, we obtain the condition
	\[
		(\delta + \delta')(p-2)/2 - (1-\gamma_2) > 0,
	\]
	which is equivalent to
	\[
		p > 2 + 2(1-\gamma_2)/(\delta + \delta').
	\]
	Since $3(1-\gamma_2) + (1-\gamma_1) + 6 \delta < 1$ we conclude
	\[
		p \geq 2 + \frac{12(1-\gamma_2)}{\gamma_1 + 3 \gamma_2 - 3}. \qedhere
	\]
\end{proof}
Next, we show Theorem \ref{thm:1} and Theorem \ref{thm:2}.
\begin{proof}[Proof of Theorem \ref{thm:1} and Theorem \ref{thm:2}]
	Let $(a_n : n \in \NN)$ be a sequence of complex numbers such that $\sup_{n \in \NN} |a_n| \le 1$.
	Using Proposition \ref{prop:2} with $f(n) = a_n \psi(n)$ we get
	\begin{align*}
		\int_{\TT}
		\Big|
		\sum_{n\in\mathbf{B}_{N}}f(n)\psi(n)^{-1} e^{2\pi i \xi n}
		\Big|^p
		{\: \rm d}\xi
		\lesssim_p
		N^{p/2-1}
		\Big(\sum_{n \in \mathbf{B}_{N}}
		| f(n) |^2 \psi(n)^{-1}
		\Big)^{p/2},
	\end{align*}
	thus, by \eqref{eq:32},
	\begin{align*}
		\int_{\TT}
		\Big|
		\sum_{n \in \mathbf{B}_{N}}
		a_n\ e^{2\pi i \xi n}
		\Big|^p
		{\: \rm d}\xi
		\lesssim_p
		N^{p/2-1}
		\Big(\sum_{n \in\mathbf{B}_{N}} \psi(n) \bigg)^{p/2}
		\lesssim_p
		N^{-1} \varphi_2(N)^p.
	\end{align*}
	Finally, we may estimate
	\begin{align*}
  		\int_{\TT}
		\Big| \sum_{n\in\mathbf{B}_{N}}e^{2\pi i \xi n} \Big|^p {\: \rm d}\xi
		\gtrsim
		\int_{|\xi|\le 1/(100N)}
		\Big| \sum_{n\in\mathbf{B}_{N}}e^{2\pi i \xi n}\Big|^p {\: \rm d}\xi
		\gtrsim
		N^{-1}\varphi_2(N)^p.
	\end{align*}
	This completes the proof.
\end{proof}

\begin{bibliography}{discrete}
	\bibliographystyle{amsplain}
\end{bibliography}

\end{document}